\documentclass[10pt]{amsart}
\usepackage{amssymb,amsthm,amsmath}

\newtheorem{prop}{Proposition}
\newtheorem{thm}{Theorem}

\theoremstyle{definition}
\newtheorem{defn}{Definition}
\newtheorem{example}{Example}

\def\A{{\mathcal A}}

\def\Z{{\mathbb Z}}
\def\Q{{\mathbb Q}}
\def\R{{\mathbb R}}
\def\C{{\mathbb C}}
\def\F{\mathcal{F}}
\def\G{\mathcal{G}}
\def\B{\mathcal{B}}
\def\O{{\mathcal O}}
\def\E{{\mathcal E}}
\def\H{{\mathcal H}}
\def\cD{{\mathcal D}}
\def\F{{\mathcal F}}
\def\AS{{\mathfrak S}}

\def\DS{{\mathfrak D}}

\def\OG{\mathrm{OG}}

\def\SO{\mathrm{SO}}

\def\HH{\mathrm{H}}
\def\CH{\mathrm{CH}}
\def\Id{\mathrm{Id}}
\def\rX{\mathrm{X}}
\def\inv{\mathrm{inv}}

\def\Par{\Pi}
\def\l{{\lambda}}
\def\x{\mathrm{x}}

\def\L{{\Lambda}}
\def\X{{\mathfrak X}}

\def\a{{\alpha}}
\def\b{{\beta}}

\def\s{{\sigma}}
\def\om{{\varpi}}
\def\ome{{\omega}}
\def\Om{{\Omega}}

\DeclareMathOperator{\Spec}{Spec}
\DeclareMathOperator{\rk}{rk}

\newcommand\bull{{\scriptscriptstyle \bullet}}
\newcommand{\dis}{\displaystyle}

\newcommand{\ssm}{\smallsetminus}

\newcommand{\hra}{\hookrightarrow}
\newcommand{\ra}{\rightarrow}

\newcommand{\dbar}{\overline{\partial}}

\newcommand{\End}{\mbox{End}}

\newcommand{\Ker}{\mbox{Ker}}
\newcommand{\Tr}{\mbox{Tr}}

\newcommand{\ov}{\overline}
\newcommand{\noin}{\noindent}
\newcommand{\wt}{\widetilde}
\newcommand{\Pf}{\mbox{Pfaffian}}
\newcommand{\wh}{\widehat}

\begin{document}

\title[Schubert polynomials and Arakelov theory]
{Schubert polynomials and Arakelov theory of orthogonal
flag varieties}
\author{Harry Tamvakis}
\date{January 5, 2010}
\thanks{The author was supported in part by National Science 
Foundation Grant DMS-0901341.}
\address{University of Maryland,
Department of Mathematics,
1301 Mathematics Building,
College Park, MD 20742, USA}
\email{harryt@math.umd.edu}

\begin{abstract}
We propose a theory of combinatorially explicit Schubert polynomials
which represent the Schubert classes in the Borel presentation of the
cohomology ring of the orthogonal flag variety $\X=\SO_N/B$.  We use
these polynomials to describe the arithmetic Schubert calculus on
$\X$. Moreover, we give a method to compute the natural arithmetic
Chern numbers on $\X$, and show that they are all rational numbers.
\end{abstract}

\maketitle 

\setcounter{section}{-1}

\section{Introduction}

\noindent
Let $V$ be a complex vector space equipped with a nondegenerate
skew-symmetric bilinear form. Let $\X$ denote the flag variety for the
symplectic group, which parametrizes flags of isotropic subspaces in
$V$. In \cite{T5}, we defined a family of {\em symplectic Schubert
polynomials} for $\X$, which represent the classes of the Schubert
varieties in the Borel presentation \cite{Bo} of the cohomology ring
of $\X$. These polynomials were applied to understand the structure of
the Gillet-Soul\'e arithmetic Chow ring of $\X$, thought of as a
smooth scheme over the ring of integers. Our aim in this companion
paper to \cite{T5} is to explain the analogous theory for the
orthogonal group, which arises when the chosen bilinear form on $V$ is
symmetric.

The symplectic Schubert polynomials of \cite{T5} are closely related
to the type C Schubert polynomials of Billey and Haiman \cite{BH}.  As
in \cite[Thm.\ 3]{BH}, our theory of orthogonal Schubert polynomials
for the root system of type $\mathrm{B}_n$ is, up to well known scalar
factors, the same as that for the root system
$\mathrm{C}_n$. Moreover, using these $\mathrm{B}_n$ Schubert
polynomials, one can describe the arithmetic Chow ring of the flag
variety of the odd orthogonal group in a similar fashion to the
symplectic group, following \cite[Thm.\ 3]{T5}.  Therefore in this
paper we will concentrate on the even orthogonal case, and construct
Schubert polynomials for the root system of type $\mathrm{D}_n$.
For the application to arithmetic intersection theory, we must
deal with an extra relation which comes from the vanishing of the top
Chern class of the maximal isotropic subbundle of the trivial vector
bundle over $\X$. Fortunately, this relation can be computed using our
work \cite{T4} on the Arakelov theory of even orthogonal
Grassmannians.

This paper is organized as follows. We begin in \S \ref{prelims} with
combinatorial preliminaries on $\wt{P}$-polynomials and the
Lascoux-Sch\"utzenberger and Billey-Haiman Schubert polynomials. We
introduce our theory of orthogonal Schubert polynomials in \S
\ref{schubdef} and list some of their basic properties in \S
\ref{sproperties}. Section \ref{hvb} computes the curvature of the
relevant homogeneous vector bundles over $\X(\C)$, equipped with their
natural hermitian metrics. The arithmetic intersection theory of $\X$
is studied in \S \ref{ait}.  Our method for computing arithmetic
intersections is explained in \S \ref{compaa}, and the arithmetic
Schubert calculus is described in \S \ref{asc}.

I wish to thank the anonymous referees, whose comments helped to 
improve the exposition in this article.

\section{Preliminary definitions}
\label{prelims}

\subsection{$\wt{P}$- and $P$-functions}
\label{definitions}
We let $\Par$ denote the set of all integer partitions.  The {\em
length} $\ell(\l)$ of a partition $\l=(\l_1,\ldots,\l_r)$ is the
number of (nonzero) parts $\l_i$, and the weight $|\l|$ is the sum
$\sum_i \l_i$. We let $\l_i=0$ for any $i>\ell(\l)$.  A partition is
{\em strict} if no nonzero part is repeated. Let $\G_n=\{\l\in \Pi \
|\ \l_1 \leq n\}$ and let $\F_n$ be the set of strict partitions in
$\G_n$. 

Let $\rX=(\x_1,\x_2,\ldots)$ be a sequence of commuting independent
variables. Define the elementary symmetric functions $e_k=e_k(\rX)$
by the generating series 
\[
\sum_{k=0}^{\infty}e_k(\rX)t^k = \prod_{i=1}^{\infty}(1+\x_it).
\]
We will often work with coefficients in the ring $A=\Z[\frac{1}{2}]$;
the polynomial ring $\Lambda' = A[e_1,e_2,\ldots]$ is the ring of
symmetric functions in the variables $\rX$ with these coefficients.
Next, we define the $\wt{P}$-functions of Pragacz and Ratajski
\cite{PR}.  Set $\wt{P}_0=1$ and $\wt{P}_k=e_k/2$ for $k> 0$.  For
$i,j$ nonnegative integers, let
\[
\wt{P}_{i,j}=
\wt{P}_i\wt{P}_j+
2\sum_{r=1}^{j-1}(-1)^r\wt{P}_{i+r}\wt{P}_{j-r}+(-1)^j\wt{P}_{i+j}.
\]
If $\l$ is a partition of length greater than two, define
\[
\dis
\wt{P}_{\l}=\Pf(\wt{P}_{\l_i,\l_j})_{1\leq i<j\leq 2m},
\]
where $m$ is the least positive integer with $2m\geq \ell(\l)$.

These $\wt{P}$-functions have the following properties:

\medskip

(a) The $\wt{P}_{\l}(\rX)$ for $\l\in\Pi$ form an $A$-basis of $\L'$.

\medskip

(b) $\wt{P}_{k,k}(\rX)= \frac{1}{4} e_k(\rX^2) = 
\frac{1}{4}e_k(\x_1^2,\x_2^2,\ldots)$ 
for all $k>0$.

\medskip

(c) If $\l=(\l_1,\ldots,\l_r)$ and
$\l^+=\l\cup(k,k)=(\l_1,\ldots,k,k,\ldots,\l_r)$ then
\[
\wt{P}_{\l^+}=\wt{P}_{k,k}\wt{P}_{\l}.
\]

\medskip

(d) The coefficients of $\wt{P}_{\l}(\rX)$ are nonnegative
rational numbers.

\medskip
\noin
Let $\L'_n=A[\x_1,\ldots,\x_n]^{S_n}$ be the ring of symmetric polynomials
in $\rX_n=(\x_1,\ldots,\x_n)$. Then we have two additional properties.

\medskip

(e) If $\l_1>n$, then $\wt{P}_{\l}(\rX_n)=0$. The 
$\wt{P}_{\l}(\rX_n)$ for $\l\in\G_n$ form an $A$-basis of $\L'_n$.

\medskip

(f) $\wt{P}_n(\rX_n)\wt{P}_\l(\rX_n) = \wt{P}_{(n,\l)}(\rX_n)$ for all
$\l\in\G_n$.

\medskip
Suppose that $Y= (y_1,y_2,\ldots)$ is a second sequence of variables
and define symmetric functions $q_k(Y)$ by the equation
\[
\sum_{k=0}^{\infty}q_k(Y)t^k = \prod_{i=1}^{\infty}\frac{1+y_it}{1-y_it}.
\]
Let $\Gamma' = A[q_1,q_2\ldots]$ and define an $A$-algebra
homomorphism $\eta:\L' \ra \Gamma'$ by setting $\eta(e_k(\rX)) =
q_k(Y)$ for each $k\geq 1$. For any strict partition $\l$, the Schur
$P$-function $P_{\l}(Y)$ may be defined as the image of
$\wt{P}_{\l}(\rX)$ under $\eta$.  The $P_\l$ for strict partitions
$\l$ have nonnegative {\em integer} coefficients and form a free
$A$-basis of $\Gamma'$.

\subsection{Divided differences and type A Schubert polynomials}
\label{Wsec}
The symmetric group $S_n$ is the Weyl group for the root system
$\text{A}_{n-1}$. We write the elements $\om$ of $S_n$ using the
single-line notation $(\om(1),\om(2),\ldots,\om(n))$. The group $S_n$
is generated by the simple transpositions $s_i$ for $1\leq i\leq n-1$,
where $s_i$ interchanges $i$ and $i+1$ and fixes all other elements of
$\{1,\ldots,n\}$.

The elements of the Weyl group $\wt{W}_n$ for the root system
$\text{D}_n$ may be represented by signed permutations; we will adopt 
the notation where a bar is written over an element with a negative sign.
The group $\wt{W}_n$ is an extension of $S_n$ by an
element $s_0$ which acts on the right by
\[
(u_1,u_2,\ldots,u_n)s_0=(\ov{u}_2,\ov{u}_1,u_3,\ldots,u_n).
\]
A {\em reduced word} of $w\in\wt{W}_n$ is a sequence $a_1\ldots a_r$
of elements in $\{0,1,\ldots,n-1\}$ such that $w=s_{a_1}\cdots
s_{a_r}$ and $r$ is minimal (so equal to the length $\ell(w)$ of
$w$). If we convert all the $0$'s which appear in the reduced word
$a_1\ldots a_r$ to $1$'s, we obtain a {\em flattened word} of $w$.
For example, $20312$ is a reduced word of $\ov{1}4\ov{3}2$, and
$21312$ is the corresponding flattened word. Note that $21312$ is 
also a word, but not reduced, for $1432$. The elements of maximal
length in $S_n$ and $\wt{W}_n$ are
\[
\om_0=(n,n-1,\ldots,1) \  \ \ \mathrm{and} \ \ \ 
w_0=\begin{cases} (\ov{1},\ov{2},\ldots,\ov{n}) & \text{if $n$ is even}, \\
(1,\ov{2},\ldots,\ov{n}) & \text{if $n$ is odd}
\end{cases} 
\]
respectively.

The group $\wt{W}_n$ acts on the ring $A[\rX_n]$ of polynomials in
$\rX_n$: the transposition $s_i$ interchanges $\x_i$ and $\x_{i+1}$
for $1\leq i\leq n-1$, while $s_0$ sends $(\x_1,\x_2)$ to
$(-\x_2,-\x_1)$ (all other variables remain fixed). Following
\cite{BGG} and \cite{D1, D2}, we have divided difference
operators $\partial_i: A[\rX_n]\ra A[\rX_n]$. For $1\leq i\leq n-1$
they are defined by
\[
\partial_i(f)=(f-s_if)/(\x_i-\x_{i+1})
\]
while 
\[
\partial_0(f)=(f-s_0 f)/(\x_1+\x_2),
\]
for any $f\in A[\rX_n]$. For each $w\in \wt{W}_n$, define an operator
$\partial_w$ by setting
\[
\partial_w=\partial_{a_1}\circ \cdots \circ \partial_{a_\ell}
\]
if $w=a_1 \cdots a_\ell$ is a reduced word for $w$.

For every permutation $\om\in S_n$, Lascoux and Sch\"utzenberger
\cite{LS} defined a {\em type A Schubert polynomial}
$\AS_{\om}(\rX_n)\in\Z[\rX_n]$ by
\[
\AS_{\om}(\rX_n)=\partial_{\om^{-1}\om_0}\left(\x_1^{n-1}\x_2^{n-2}\cdots
\x_{n-1} \right).
\]
This definition is stable under the natural inclusion of $S_n$ into
$S_{n+1}$, hence the polynomial $\AS_w$ makes sense for $w\in
S_{\infty}= \cup_{n=1}^\infty S_n$. The $\AS_w$ for $w \in
S_{\infty}$ form a $\Z$-basis of $\Z[\rX]=\Z[\x_1,\x_2,\ldots]$.
The coefficients of $\AS_w$ are nonnegative integers.

\subsection{Billey-Haiman Schubert polynomials of type D}
We regard $\wt{W}_n$ as a subgroup of $\wt{W}_{n+1}$ in the obvious
way and let $\wt{W}_\infty$ denote the union of all the $\wt{W}_n$.
Let $Z=(z_1,z_2,\ldots)$ be a third sequence of commuting variables.
Billey and Haiman \cite{BH} defined a family $\{\cD_w\}_{w\in
\wt{W}_\infty}$ of Schubert polynomials of type D, which form an
$A$-basis of the ring $\Gamma'[Z]$. The expansion coefficients for a
product $\cD_u\cD_v$ in the basis of type D Schubert polynomials agree
with the Schubert structure constants on even orthogonal flag
varieties for sufficiently large $n$. For every $w\in \wt{W}_n$ there
is a unique expression
\begin{equation}
\label{bheq}
\cD_w = \sum_{{\l \, \text{strict}}\atop{\om\in S_n}}f_{\l,\om}^w
P_\l(Y)\AS_\om(Z)
\end{equation}
where the coefficients $f_{\l,\om}^w$ are nonnegative integers. 
We proceed to give a combinatorial formula for these 
numbers.

A sequence $a=(a_1,\ldots,a_m)$ is called {\em unimodal} if for some
$r \leq m$, we have
\[
a_1 > a_2 > \cdots > a_{r-1} \geq a_r < a_{r+1} < \cdots < a_m,
\]
and if $a_{r-1}=a_r$ then $a_r=1$.

Let $w\in \wt{W}_n$ and $\l$ be a Young diagram with $r$ rows such that
$|\l| = \ell(w)$. A {\em Kra\'skiewicz-Lam tableau} for $w$ of shape $\l$
is a filling $T$ of the boxes of $\lambda$ with positive integers
in such a way that

\medskip
\noindent
a) If $t_i$ is the sequence of entries in the $i$-th row of $T$,
reading from left to right, then the row word $t_r\ldots t_1$ is
a flattened word for $w$.

\medskip
\noindent
b) For each $i$, $t_i$ is a unimodal subsequence of maximum length
in $t_r \ldots t_{i+1} t_i$.

\medskip
\noin 
Let $T$ be a Kra\'skiewicz-Lam tableau of shape $\l$ with row
word $a_1\ldots a_\ell$. We define $m(T) = \ell(\l)+1-k$, where
$k$ is the number of distinct values of $s_{a_1}\cdots s_{a_j}(1)$ for
$0\leq j\leq \ell$. It follows from \cite[Thm.\ 4.35]{La} that $m(T)\geq 0$.

\begin{example} 
Let $\l\in\F_{n-1}$, $\ell=\ell(\l)$, $k=n-1-\ell$, and $\mu$ be the
strict partition whose parts are the numbers from $1$ to $n$ which do
not lie in the set $\{1,\l_\ell+1,\ldots,\l_1+1\}$. The barred
permutation
\[
w_{\l}=(\ov{\l_1+1},\ldots,\ov{\l_{\ell}+1},\hat{1},\mu_k,\ldots,\mu_1)
\]
where $\hat{1}$ is equal to $1$ or $\ov{1}$ according to the parity of
$\ell$ is the {\em maximal Grassmannian element} of $\wt{W}_n$
corresponding to $\l$. There is a unique Kra\'skiewicz-Lam tableau
$T_\l$ for $w_{\l}$, which has shape $\l$, and whose $i$-th row
consists of the numbers $1$ through $\l_i$ in decreasing
order. Moreover, we have $m(T_\l)=0$. For instance, if $\l = (6,4,3)$
then we obtain
\[
T_\l \ = \ 
\begin{array}{l}
6 \ 5 \ 4 \ 3 \ 2 \ 1 \\
4 \ 3 \ 2 \ 1 \\
3 \ 2 \ 1.
\end{array}
\]
\end{example}

\begin{prop}[BH, La]
\label{BHL}
For every $w \in \wt{W}_\infty$, we have $f^w_{\l,\om}= \sum_T
2^{m(T)}$, summed over all Kra\'skiewicz-Lam tableaux $T$ for
$w\om^{-1}$ of shape $\l$, if $\ell(w\om^{-1}) = \ell(w) - \ell(\om)$,
and $f^w_{\l,\om}=0$ otherwise.
\end{prop}
\begin{proof}
According to \cite[Thm.\ 3]{BH}, the polynomial $\cD_w$ satisfies
\[
\cD_w = \sum_{uv=w} E_u(Y) \AS_v(Z),
\]
summed over all factorizations $uv=w$ in $\wt{W}_\infty$ such that
$\ell(u)+\ell(v) = \ell(w)$ with $v\in S_\infty$. The left factors
$E_u(Y)$ are the type D Stanley symmetric functions of \cite{BH, La}.
We deduce from \cite[Prop.\ 3.7]{BH} and \cite[Thm.\ 4.35]{La} that
for any $u\in \wt{W}_\infty$,
\[
E_u(Y) = \sum_{\l} d^u_{\l} \, P_{\l}(Y)
\]
where $d^u_{\l}= \sum_T 2^{m(T)}$, summed over all Kra\'skiewicz-Lam
tableaux $T$ for $u$ of shape $\l$.  The result follows by combining
these two facts.
\end{proof}

\section{Orthogonal Schubert polynomials}
\label{ssp}

\subsection{}
\label{soflag}
Consider the vector space $\C^{2n}$ with its canonical basis
$\{e_i\}_{i=1}^{2n}$ of unit coordinate vectors. We define the
{\em skew diagonal symmetric form} $[\ \,,\ ]$ on $\C^{2n}$ by setting
$[e_i,e_j]=0$ for $i+j\neq 2n+1$ and $[e_i,e_{2n+1-i}]=1$
for $1\leq i \leq 2n$. The orthogonal group $\SO_{2n}(\C)$ is 
the group of linear automorphisms of $\C^{2n}$ preserving 
the symmetric form. The upper triangular matrices in $\SO_{2n}$
form a Borel subgroup $B$.

A subspace $\Sigma$ of $\C^{2n}$
is called isotropic if the restriction of the symmetric 
form to $\Sigma$ vanishes. Consider a partial flag of 
subspaces 
\[
0= E_0 \subset E_1 \subset \cdots \subset E_n \subset E_{2n} = \C^{2n}
\]
with $\dim E_i = i$ and $E_n$ isotropic. Each such flag can be
extended to a complete flag $E_{\bull}$ in $\C^{2n}$ by letting
$E_{n+i}=E_{n-i}^{\perp}$ for $1\leq i\leq n$; we will call such a
flag a {\em complete isotropic flag}. We say that two isotropic
subspaces $E$ and $F$ of dimension $n$ are {\em in the same family} if
$\dim(E\cap F) \equiv n \,(\text{mod } 2)$; two complete isotropic
flags $E_\bull$ and $F_\bull$ are in the same family if $E_n$ and
$F_n$ are. The variety $\X=\SO_{2n}/B$ parametrizes complete isotropic
flags $E_\bull$ with $E_n$ in the same family as $\langle e_1,\ldots
,e_n\rangle$.  We use the same notation to denote the tautological
flag $E_\bull$ of vector bundles over $\X$.

There is a group monomorphism $\phi:\wt{W}_n\hra S_{2n}$ whose image
consists of those permutations 
$\om\in S_{2n}$ such that $\om(i)+\om(2n+1-i) = 2n+1$ for all $i$
and the number of $i\leq n$ such that $\om(i)>n$ is even.
The map $\phi$ is determined by setting, 
for each $w=(w_1,\ldots,w_n)\in \wt{W}_n$ and $1\leq i \leq n$, 
\[
\phi(w)(i)=\left\{ \begin{array}{cl}
             n+1-w_{n+1-i} & \mathrm{ if } \ w_{n+1-i} \ \mathrm{is} \ 
             \mathrm{unbarred}, \\
             n+\ov{w}_{n+1-i} & \mathrm{otherwise}.
             \end{array} \right.
\]

Let $F_{\bull}$ be a fixed complete isotropic flag in the same family
as the flags in $\X$.  For every $w\in \wt{W}_n$ define the {\em Schubert
variety} $\X_w(F_\bull)\subset \X$ as the closure of the locus of
$E_\bull \in \X$ such that
\[
\dim(E_r\cap F_s) =  \#\,\{\,i \leq r \ |\ \phi(w_0ww_0)(i)> 2n-s\,\}
\ \ \mathrm{for} \ \ 1\leq r\leq n-1,\, 1\leq s\leq 2n.
\]
The Schubert class $\sigma_w$ in $\HH^{2\ell(w)}(\X,\Z)$ is the 
cohomology class which is Poincar\'e dual to the homology class
determined by $\X_w(F_\bull)$.

Following Borel \cite[\S 29]{Bo}, the cohomology ring $\HH^*(\X,A)$ is
presented as a quotient
\begin{equation}
\label{presentation}
\HH^*(\X,A) \cong A[\x_1,\ldots,\x_n]/J_n
\end{equation}
where $J_n$ is the ideal generated by the $\wt{W}_n$-invariants of
positive degree in $A[\rX_n]$. 
The inverse of the isomorphism (\ref{presentation}) sends the class of
$\x_i$ to $-c_1(E_{n+1-i}/E_{n-i})$ for each $i$ with $1\leq i \leq
n$.

\subsection{}
\label{schubdef}
For every $\l\in \G_n$ and $\om\in S_n$, define the polynomial 
$\DS_{\l,\om}=\DS_{\l,\om}(\rX_n)$ by
\[
\DS_{\l,\om} = \wt{P}_{\l}(\rX_n) \AS_\om(-\rX_n) = 
(-1)^{\ell(\om)}\wt{P}_{\l}(\rX_n) \AS_\om(\rX_n).
\]
Lascoux and Pragacz \cite{LP} showed that the products
$\wt{P}_{\l}(\rX_n) \AS_\om(\rX_n)$ for $\l\in\F_{n-1}$ and $\om\in
S_n$ form a basis for the polynomial ring $A[\rX_n]$ as an
$A[\rX_n]^{\wt{W}_n}$-module.  Observe that the $\DS_{\l,\om}(\rX_n)$
for $\l\in\G_n$ and $\om\in S_n$ form a basis of $A[\x_1,\ldots,\x_n]$
as an $A$-module.  The ideal $J_n$ of \S \ref{soflag} is generated by
the polynomials $e_i(\rX_n^2)= 4\, \wt{P}_{i,i}(\rX_n)$ and
$e_n(\rX_n)=2\,\wt{P}_n(\rX_n)$, and the $\wt{P}$-polynomials have the
factorization properties (c), (f) and the vanishing property (e) of
\S \ref{definitions}. We deduce that $\wt{P}_\l(\rX_n)\in J_n$
unless $\l\in \F_{n-1}$.

\begin{defn}
\label{cdefine}
For $w\in \wt{W}_n$, define the orthogonal Schubert polynomial 
$\DS_w=\DS_w(\rX_n)$ by
\[
\DS_w = \sum_{{\l\in \F_{n-1}}\atop{\om\in S_n}} 
f^w_{\l,\om} \DS_{\l,\om}(\rX_n)
\]
where the coefficients $f^w_{\l,\om}$ are the same 
as in (\ref{bheq}) and Proposition \ref{BHL}.
\end{defn}

\begin{thm}
\label{defthm}
The orthogonal Schubert polynomial $\DS_w(\rX_n)$ 
is the unique $\Z$-linear combination of the $\DS_{\l,\om}(\rX_n)$ for 
$\l\in\F_{n-1}$ and $\om\in S_n$ which represents the Schubert class
$\sigma_w$ in the Borel presentation {\em (\ref{presentation})}.
\end{thm}
\begin{proof}
Recall that a partition is {\em
odd} if all its non-zero parts are odd integers. For each partition
$\mu$, let $p_\mu=\prod_ip_{\mu_i}$, where $p_r(\rX) =
\x_1^k+\x_2^k+\cdots$ denotes the $r$-th power sum. The $p_\mu(Y)$ for
$\mu$ odd form a $\Q$-basis of $\Gamma'\otimes_A\Q$. We therefore have
a unique expression
\begin{equation}
\label{ontheway}
\cD_w = \sum_{{\mu\, \text{odd}}\atop{\om\in S_n}}a_{\mu,\om}^w\,
p_{\mu}(Y)\AS_\om(Z)
\end{equation}
in the ring $\Gamma'[Z]\otimes_A\Q$. 

J\'ozefiak \cite{J} showed that the kernel of the homomorphism $\eta$
from \S \ref{definitions} is the ideal generated by the symmetric
functions of positive degree in $\rX^2=(\x_1^2,\x_2^2,\ldots)$. 
It follows from this and properties (b), (c) of \S \ref{definitions} that
$\eta(\wt{P}_{\l}) = 0$ unless $\l$ is a strict partition. Moreover,
we have $\eta(p_k(\rX)) = 2\,p_k(Y)$, if $k$ is odd, and
$\eta(p_k(\rX))=0$, if $k>0$ is even.

Let $p_{odd}=(p_1,p_3,p_5,\ldots)$.  Define a polynomial
$\cD_w(p_{odd}(\rX), \rX_{n-1})$ in the variables $p_k:=p_k(\rX)$ for
$k$ odd and $\x_1,\ldots,\x_{n-1}$ by substituting $p_k(Y)$ with
$p_k(\rX)/2$ and $z_i$ with $-\x_i$ in (\ref{ontheway}). 
We deduce from (\ref{bheq}), (\ref{ontheway}), and the above discussion 
that $\cD_w(p_{odd}(\rX),\rX_{n-1})$ differs from
\[
\sum_{{\l \, \text{strict}}\atop{\om\in S_n}} f_{\l,\om}^w
\wt{P}_\l(\rX)\AS_\om(-\rX_n)
\]
by an element in the ideal of $\L'[\rX_{n-1}]$ generated by the
$e_i(\rX^2)$ for $i>0$. 

According to \cite[\S 2]{BH}, for every $w\in \wt{W}_n$, the
polynomial 
\[
\cD_w(\rX_n):= \cD_w(p_{odd}(\rX_n),\rX_{n-1})
\] 
obtained by setting $\x_i = 0$ for all $i>n$ in
$\cD_w(p_{odd}(\rX),\rX_{n-1})$ represents the Schubert class
$\sigma_w$ in the Borel presentation (\ref{presentation}).  Since
$\wt{P}_\l(\rX_n)\in J_n$ unless $\l\in \F_{n-1}$, it follows that
$\DS_w$ represents the Schubert class $\s_w$ in the presentation
(\ref{presentation}), as required.

We claim that the $\DS_{\l,\om}$ for $\l\in \G_n\ssm\F_{n-1}$ and
$\om\in S_n$ form an $A$-basis of $J_n$. To see this, note that if $h$
is an element of $J_n$ then $h(\rX_n)=\sum_ie_i(\rX_n^2)f_i(\rX_n)
+e_n(\rX_n)g(\rX_n)$ for some polynomials $f_i,g\in A[\rX_n]$. Now the
$f_i$ and $g$ are unique $A$-linear combinations of the
$\DS_{\mu,\om}$ for $\mu\in\G_n$ and $\om\in S_n$, and properties (b),
(c), and (f) of \S \ref{definitions} give
\[
e_i(\rX_n^2)\DS_{\mu,\om}(\rX_n) =
4\,\DS_{\mu\cup(i,i),\om}(\rX_n)
\]
and 
\[ 
e_n(\rX_n)\DS_{\mu,\om}(\rX_n) = 2\,\DS_{(n,\mu),\om}(\rX_n),
\]
respectively. We deduce that any $h\in J_n$ lies in the $A$-linear
span of the $\DS_{\l,\om}$ for $\l\in \G_n\ssm\F_{n-1}$ and $\om\in S_n$.
Since the $\DS_{\l,\om}$ for $\l\in \G_n$ and $\om\in S_n$ are
linearly independent, this proves the claim and the uniqueness
assertion in the theorem.
\end{proof}

\medskip

The statement of Theorem \ref{defthm} may serve as an
alternative definition of the orthogonal Schubert polynomials
$\DS_w(\rX_n)$.

\subsection{}
\label{sproperties}
We give below some properties of the polynomials $\DS_w(\rX_n)$.

\medskip

(a) The set
\[
\{\DS_w \ |\ w\in \wt{W}_n\}
\cup \{\DS_{\l,\om}\ |\ \l\in \G_n\ssm\F_{n-1},\ \om\in S_n\}
\]
is an $A$-basis of the polynomial ring $A[\x_1,\dots,x_n]$. The
$\DS_{\l,\om}$ for $\l\in \G_n\ssm\F_{n-1}$ and $\om\in S_n$ span the
ideal $J_n$ of $A[\x_1,\ldots,\x_n]$ generated by the $e_i(\rX_n^2)$
for $1\leq i \leq n-1$ and $e_n(\rX_n)=\x_1\cdots\x_n$.

\medskip

(b) For every $u,v\in \wt{W}_n$, we have an equation
\begin{equation}
\label{structeq}
\DS_u\cdot \DS_v = \sum_{w\in \wt{W}_n}d_{uv}^w\,\DS_w + 
\sum_{{\l\in\G_n\ssm\F_{n-1}}\atop{\om\in S_n}}
d_{uv}^{\l\om}\,\DS_{\l,\om}
\end{equation}
in the ring $A[\x_1,\ldots,\x_n]$. The coefficients $d_{uv}^w$ are 
nonnegative integers, which vanish unless $\ell(w)=\ell(u)+\ell(v)$, 
and agree with the structure constants in the equation of 
Schubert classes 
\[
\s_u\cdot \s_v = \sum_{w\in \wt{W}_n}d_{uv}^w\, \s_w,
\]
which holds in $\HH^*(\X,\Z)$. The coefficients $d_{uv}^{\l\om}$ are
integers, some of which may be negative. Equation (\ref{structeq})
provides a lifting of the Schubert calculus from the cohomology ring
$\HH^*(\X,A)\cong A[\x_1,\ldots,\x_n]/J_n$ to the polynomial ring 
$A[\x_1,\ldots,\x_n]$.

\medskip

(c) For each $m<n$ let $i=i_{m,n}:\wt{W}_m \to \wt{W}_n$ 
be the natural embedding using the first $m$ components. Then for any
$w\in \wt{W}_m$ we have
\[
\left.
\DS_{i(w)}(\rX_n)\right|_{x_{m+1}=\cdots=x_n=0}\, =\, \DS_w(\rX_m).
\]

(d) For $\om\in S_n$ and $w\in \wt{W}_n$, we have
\[
\partial_\om\DS_w = \begin{cases}
(-1)^{\ell(\om)}\,\DS_{w\om} & \text{if} \ \ell(w\om) = \ell(w) - \ell(\om), \\
0 & \text{otherwise}.
\end{cases}
\]

\medskip
\medskip
\noin
The remaining properties listed in \cite[\S 2.3]{T5} also have
analogues here, and their proofs are similar.

\begin{example} a) We have the equations
\begin{gather*}
\DS_{s_0}(\rX_n) = \wt{P}_1(\rX_n) = \frac{1}{2}(x_1+x_2+\cdots + x_n) \\
\DS_{s_1}(\rX_n) = \wt{P}_1(\rX_n) - \AS_{s_1}
= \frac{1}{2}(-x_1+x_2+\cdots + x_n) \\
\DS_{s_i}(\rX_n) = 2\,\wt{P}_1(\rX_n) - \AS_{s_i}
= x_{i+1}+\cdots + x_n \ \  \text{for} \ \ 2 \leq i \leq n-1.
\end{gather*}
\medskip
\noin
b) For a maximal Grassmannian element $w_\l\in \wt{W}_n$, we have 
$\DS_w(\rX_n) = \wt{P}_\l(\rX_n)$. 
\end{example}

\begin{example}
The list of all orthogonal Schubert polynomials $\DS_w$ for $w\in
\wt{W}_3$ is given in Table \ref{schubtable}. These polynomials are
displayed according to the four orbits of the symmetric group $S_3$ on
$\wt{W}_3$. Once the highest degree term in each orbit is known, one
can compute the remaining elements easily using type A divided
differences, by property (d) above. The reader should compare this
table with \cite[Table 3]{BH}.
\end{example}

{\small{
\begin{table}[t]
\caption{Orthogonal Schubert polynomials for $w\in \wt{W}_3$}
\centering
\begin{tabular}{|l|c|} \hline
$w$ & $\DS_w(\rX_3)=\sum f_{\l,\om}^w
\,\wt{P}_{\l}(\rX_3)\,\AS_\om(-\rX_3)$ \\ \hline 
$123 = 1$ & $1$ \\
$213 = s_1$ & $\wt{P}_1 - \AS_{213}$ \\ 
$132 = s_2$ & $2\,\wt{P}_1 - \AS_{132}$ \\ 
$231 = s_1s_2$ & $\wt{P}_2 - \wt{P}_1\,\AS_{132} + \AS_{231}$ \\
$312=s_2s_1$ & $\wt{P}_2 - 2\,\wt{P}_1\,\AS_{213} + \AS_{312}$ \\ 
$321 = s_1s_2s_1$ & $\wt{P}_{21}
-\wt{P}_2\,\AS_{213}-\wt{P}_2\,\AS_{132}
+\wt{P}_1\,\AS_{312}+2\,\wt{P}_1\,\AS_{231} - \AS_{321}$ \\

$\ov{2}\ov{1}3 = s_0$ & $\wt{P}_1$ \\
$\ov{1}\ov{2}3 = s_0 s_1$ & $\wt{P}_2 - \wt{P}_1\,\AS_{213}$ \\ 
$\ov{2}3\ov{1} = s_0 s_2$ & $\wt{P}_2 - \wt{P}_1\,\AS_{132}$ \\
$\ov{1}3\ov{2}=s_0 s_1s_2$ & $-\wt{P}_2\,\AS_{132}
+\wt{P}_1\,\AS_{231}$ \\
$3\ov{2}\ov{1} = s_0 s_2s_1$ & $\wt{P}_{21}-\wt{P}_2\,\AS_{213}
+\wt{P}_1\,\AS_{312}$ \\
$3\ov{1}\ov{2}=s_0 s_1s_2s_1$ & $-\wt{P}_{21}\,\AS_{132}
+\wt{P}_2\,\AS_{312}+\wt{P}_2\,\AS_{231}-\wt{P}_1 \,\AS_{321}$ \\

$\ov{3}\ov{1}2 = s_2s_0$ & $\wt{P}_2$ \\
$\ov{1}\ov{3}2 = s_2s_0 s_1$ & $-\wt{P}_2\,\AS_{213}$ \\
$\ov{3}2\ov{1} = s_2s_0 s_2$ & $\wt{P}_{21} - \wt{P}_2\,\AS_{132}$ \\
$\ov{1}2\ov{3} = s_2s_0 s_1s_2$ & $\wt{P}_2\,\AS_{231}$ \\
$2\ov{3}\ov{1} = s_2s_0 s_2s_1$ & $-\wt{P}_{21}\,\AS_{213}
+\wt{P}_2\,\AS_{312}$ \\
$2\ov{1}\ov{3} = s_2s_0 s_1s_2s_1$ & $\wt{P}_{21}\,\AS_{231}
-\wt{P}_2\,\AS_{321}$ \\

$\ov{3}\ov{2}1 = s_1s_2s_0$ & $\wt{P}_{21}$ \\
$\ov{2}\ov{3}1 = s_1s_2s_0 s_1$ & $-\wt{P}_{21}\,\AS_{213}$ \\
$\ov{3}1\ov{2} = s_1s_2s_0 s_2$ & $-\wt{P}_{21}\,\AS_{132}$ \\
$\ov{2}1\ov{3} = s_1s_2s_0 s_1s_2$ & $\wt{P}_{21}\,\AS_{231}$ \\
$1\ov{3}\ov{2} = s_1s_2s_0 s_2s_1$ & $\wt{P}_{21}\,\AS_{312}$ \\
$1\ov{2}\ov{3} = s_1s_2s_0 s_1s_2s_1$ & $-\wt{P}_{21}\,\AS_{321}$ \\
\hline
\end{tabular} 
\label{schubtable}
\end{table}}}

\section{Curvature of homogeneous vector bundles}
\label{hvb}

For any complex manifold $X$, we denote the space of $\C$-valued
smooth differential forms of type $(p,q)$ on $X$ by $\A^{p,q}(X)$.  A
hermitian vector bundle on $X$ is a pair $\ov{E}=(E,h)$ consisting of
a holomorphic vector bundle $E$ over $X$ and a hermitian metric $h$ on
$E$. Let $K(\ov{E})\in \A^{1,1}(X,\End(E))$ be the curvature of
$\ov{E}$ with respect to the hermitian holomorphic connection on
$\ov{E}$ and set $K_E = \frac{i}{2\pi}K(\ov{E})$.  For any
integer $k$ with $1\leq k\leq \rk(E)$, we have a Chern form
$c_k(\ov{E}):=\Tr(\bigwedge^k K_E) \in \A^{k,k}(X)$. The total
Chern form of $\ov{E}$ is $c(\ov{E}) = 1+\sum_{k=1}^n
c_k(\ov{E})$. These differential forms are closed and their classes
in the de Rham cohomology of $X$ are the Chern classes of $E$.

To simplify the notation in this section, we will redefine the group
$\SO_{2n}(\C)$ using the {\em standard symmetric form} $[\ \,,\ ]'$
on $\C^{2n}$ whose matrix $[e_i,e_j ]_{i,j}'$ on unit coordinate
vectors is $\left(\begin{array}{cc} 0 & \Id_n \\ \Id_n & 0
\end{array}\right)$, where $\Id_n$ denotes the $n\times n$ identity
matrix.  Let $\X=\SO_{2n}/B$ be the orthogonal flag variety and
$E_\bull$ its tautological complete isotropic flag of vector bundles.
We equip the trivial vector bundle $E_{2n}=\C^{2n}_\X$ with the
trivial hermitian metric $h$ compatible with the symmetric form $[\,\
,\ ]'$ on $\C^{2n}$.

The metric $h$ on $E$ induces metrics on all the subbundles $E_i$ and the
quotient line bundles $Q_i=E_i/E_{i-1}$, for $1\leq i\leq n$.  Our
goal here is to compute the $\SO(2n)$-invariant curvature matrices of
the homogeneous vector bundles $\ov{E}_i$ and $\ov{Q}_i$ for $1\leq i
\leq n$.  As in \cite[\S 3.2]{T5}, we do this by pulling
back these matrices of $(1,1)$-forms from $\X$ to the compact Lie
group $\SO(2n)$, where their entries may be expressed in terms of the
basic invariant forms on $\SO(2n)$.

The Lie algebra of $\SO_{2n}(\C)$ is given by 
$${\mathfrak s}{\mathfrak o}(2n,\C)=\{(A,B,C)\ |\ A,B,C \in
M_n(\C), \ B,C \mbox{ skew symmetric} \},$$ 
where $(A,B,C)$ denotes the matrix 
$\left(
\begin{array}{cc}
  A & B \\
  C & -A^t
\end{array}
\right)$.  Complex conjugation of the algebra ${\mathfrak s}{\mathfrak
o}(2n,\C)$ with respect to the Lie algebra of $\SO(2n)$ is given by
the map $\tau$ with $\tau(A)=-\ov{A}^t$.  The Cartan subalgebra
$\mathfrak{h}$ consists of all matrices of the form
$\{(\mathrm{diag}(t_1,\ldots,t_n),0,0)\ |\ t_i\in \C\}$, where
$\mathrm{diag}(t_1,\ldots,t_n)$ denotes a diagonal matrix. Consider
the set of roots
\[
R = \{ \pm t_i \pm t_j\ |\ i \neq j \} \subset {\mathfrak h}^*
\]
and a system of positive roots 
\[
R^+=\{t_i-t_j\ |\ i<j\}\cup \{ t_p+t_q\ |\ p < q \},
\]
where the indices run from $1$ to $n$.  We use $ij$ to denote a
positive root in the first set and $pq$ for a positive root in the
second. The corresponding basis vectors are $e_{ij}=(E_{ij},0,0)$ and
$e^{pq}=(0,E_{pq}-E_{qp},0)$ for $p< q$, where $E_{ij}$ is the matrix
with $1$ as the $ij$-th entry and zeroes elsewhere.

Define $\ov{e}_{ij} = \tau(e_{ij})$, $\ov{e}^{pq} = \tau(e^{pq})$, and
consider the linearly independent set 
\[
\B'=\{e_{ij},\, \ov{e}_{ij},\, e^{pq},\, \ov{e}^{pq}\ |\ i<j ,\ p < q\}.
\]  
The adjoint representation of ${\mathfrak h}$ on ${\mathfrak
s}{\mathfrak o}(2n,\C)$ gives a root space decomposition
\[
{\mathfrak s}{\mathfrak o}(2n,\C) =
{\mathfrak h}\oplus \sum_{i<j}(\C\, e_{ij}\oplus \C \,\ov{e}_{ij})
\oplus \sum_{p < q}(\C \,e^{pq}\oplus \C \,\ov{e}^{pq}).
\]
Extend $\B'$ to a basis $\B$ of ${\mathfrak s}{\mathfrak o}(2n,\C)$ and
let $\B^*$ denote the dual basis of ${\mathfrak s}{\mathfrak
o}(2n,\C)^*$. Let $\ome^{ij}$, $\ov{\ome}^{ij}$, $\ome_{pq}$,
$\ov{\ome}_{pq}$ be the vectors in $\B^*$ which are dual to $e_{ij}$,
$\ov{e}_{ij}$, $e^{pq}$, $\ov{e}^{pq}$, respectively; we regard these
elements as left invariant complex one-forms on $\SO(2n)$. If $p>q$ we
agree that $\ome_{pq} = -\ome_{qp}$ and $\ov{\ome}_{pq} = -\ov{\ome}_{qp}$.
Finally, define $\ome_{ij}=\gamma\ome^{ij}$,
$\ov{\ome}_{ij}=\gamma\ov{\ome}^{ij}$, $\ome^{pq}=\gamma\ome_{pq}$,
and $\ov{\ome}^{pq}=\gamma\ov{\ome}_{pq}$, where $\gamma$ is a
constant such that $\gamma^2=\frac{i}{2\pi}$, and set $\Om_{ij} =
\ome_{ij}\wedge \ov{\ome}_{ij}$ and $\Om ^{pq} =
\ome^{pq}\wedge\ov{\ome}^{pq}$.

If $\pi:\SO(2n)\to \X$ denotes the quotient map, the pullbacks of the
aforementioned curvature matrices under $\pi$ can now be written
explicitly, following \cite[$(4.13)_X$]{GrS} and \cite[\S 3.2]{T5}. In
this way we arrive at the following proposition.

\begin{prop}
\label{grsprop} For every $k$ with $1\leq k \leq n$ we have
\[
c_1(\ov{Q}_k)=\sum_{i<k}\Omega_{ik}-
\sum_{j>k}\Omega_{kj}-
\sum_{p \neq k}\Omega^{pk}
\]
and $K_{E_k}=\{\Theta_{\a\b}\}_{1\leq \a,\b\leq k}$, where
\[
\dis
\Theta_{\a\b}=-\sum_{j>k}\ome_{\a j}\wedge\ov{\ome}_{\b j}
-\sum_{p\neq \a,\b}\ome^{p\a}\wedge\ov{\ome}^{p\b}.
\]
\end{prop}

\medskip

 Let $\dis\Omega=\bigwedge_{i<j} \Omega_{ij}\wedge\bigwedge_{p < q}
\Omega^{pq}$. It follows for instance from \cite[Cor.\ 5.16]{PR} that
the class of a point in $\X$ is Poincar\'e dual to $\dis
\frac{1}{2^{n-1}}\prod_{k=1}^{n-1} c_1(\ov{Q}_k^*)^{2n-2k}$.  We
conclude that $\dis\int_\X\Omega= \prod_{k=1}^{n-1}\frac{2}{(2k)!}$.

\section{Arithmetic intersection theory on $\SO_{2n}/B$}
\label{ait}

\subsection{Orthogonal flag varieties over $\Spec\Z$}
\label{classgps}

For the rest of this paper, $\X$ will denote the Chevalley scheme over
$\Z$ for the homogeneous space $\SO_{2n}/B$ described in \S
\ref{soflag}. Over any base field, the scheme $\X$ parametrizes
complete isotropic flags $E_\bull$ of a $2n$-dimensional vector space
$E$ equipped with the skew diagonal symmetric form, with $E_n$ in the
same family as $\langle e_1,\ldots,e_n \rangle$.  The arithmetic
orthogonal flag variety $\X$ is smooth over $\Spec \Z$, and has a
decomposition into Schubert cells induced by the Bruhat decomposition
of $\SO_{2n}$ (see e.g.\ \cite[\S 13.3]{Ja} for details).

There is a tautological complete isotropic flag of vector bundles
\[
E_\bull :\ 0=E_0\subset E_1\subset\cdots\subset E_{2n}=E
\]
over $\X$. For each $i$ with $1\leq i\leq 2n$ we let $\E_i$ denote the
short exact sequence
\[
\E_i\ :\ 0 \to E_{i-1}\to E_i \to Q_i\to 0.
\]
Let $\CH(\X)$ be the Chow ring of algebraic cycles on $\X$ modulo
rational equivalence, with coefficients in the ring $A$. Since $\X$
has a cellular decomposition, the class map induces an isomorphism
$\CH(\X)\cong \HH^*(\X(\C),A)$, following \cite[Ex.\ 19.1.11]{Fu} and
\cite[Lem.\ 6]{KM}.

We deduce that there is a ring isomorphism 
\[
\CH(\X)\cong A[\rX_n]/J_n.
\] 
This presentation of $\CH(\X)$ may be understood
geometrically as follows. The Whitney sum formula applied to the
filtration $E_\bull$ gives a Chern class equation
\[
\prod_{i=1}^{2n}(1+c_1(Q_i))=c(E)
\]
in $\CH(\X)$, which maps to the
identity $\prod_{i=1}^{2n}(1-\x_i^2) = 1$, since $E$ is a trivial
bundle. We thus obtain the relations $e_i(\rX_n^2)$ in $J_n$, for $1
\leq i \leq n-1$. Moreover, the relation $\x_1\cdots\x_n$ holds
because the top Chern class $c_n(E_n)$ vanishes.

We have an isomorphism of abelian groups 
\[
\CH(\X) \cong \bigoplus_{w\in \wt{W}_n}A\,\DS_w(\rX_n)
\]
where the polynomial $\DS_w(\rX_n)$ represents the class of the
codimension $\ell(w)$ Schubert scheme $\X_w$ in $\X$. The latter
is defined as the closure of the corresponding Schubert cell, the
complex points of which are given in \S \ref{soflag}.

\subsection{The arithmetic Chow group}
\label{acg}

For $p\geq 0$ we let $\wh{\CH}^p(\X)'$ denote the $p$-th arithmetic
Chow group of $\X$, as defined by Gillet and Soul\'e \cite{GS1}. As in
the case of $\CH(\X)$, we require coefficients in the ring $A$, so
we will work throughout with the groups $\wh{\CH}^p(\X):=
\wh{\CH}^p(\X)'\otimes_\Z A$. The elements in $\wh{\CH}^p(\X)$ are
represented by arithmetic cycles $(Z,g_Z)$, where $Z$ is a codimension
$p$ cycle on $\X$ and $g_Z$ is a current of type $(p-1,p-1)$ such that
the current $dd^cg_Z+\delta_{Z(\C)}$ is represented by a smooth
differential form on $\X(\C)$. Define
$\wh{\CH}(\X)=\bigoplus_p\wh{\CH}^p(\X)$.

Let $\A(\X(\C))=\bigoplus_p \A^{p,p}(\X(\C))$ and 
$\A'(\X(\C))\subset \A(\X(\C))$ be the
set of forms $\varphi$ in $\A(\X(\C))$ which can be written as
$\varphi=\partial \eta + \dbar \eta'$ for some smooth forms 
$\eta$, $\eta'$. Define $\wt{\A}(\X(\C))=\A(\X(\C))/\A'(\X(\C))$.
We let $F_{\infty}$ be the involution of $\X(\C)$ induced by complex 
conjugation. Let $\A^{p,p}(\X_{\R})$ be the subspace of $\A^{p,p}(\X(\C))$ 
generated by real forms $\eta$ such that $F^*_{\infty}\eta=(-1)^p\eta$;
denote by $\wt{\A}^{p,p}(\X_{\R})$
the image of $\A^{p,p}(\X_{\R})$ in $\wt{\A}^{p,p}(\X(\C))$.
Finally, let $\A(\X_{\R})=\bigoplus_p \A^{p,p}(\X_{\R})$ and
$\wt{\A}(\X_{\R})=\bigoplus_p \wt{\A}^{p,p}(\X_{\R})$.

Since the homogeneous space $\X$ admits a cellular decomposition, it
follows as in \cite{KM} that there is an exact sequence
\begin{equation}
\label{seseq}
0 \longrightarrow  \wt{\A}(\X_{\R})
\stackrel{a}\longrightarrow \wh{\CH}(\X)
\stackrel{\zeta}\longrightarrow \CH(\X)\longrightarrow 0
\end{equation}
where the maps $a$ and $\zeta$ are defined by
\[
\dis
a(\eta)=(0,\eta) \qquad \textrm{ and } \qquad
\zeta(Z,g_Z)=Z.
\]

We equip $E(\C)$ with the trivial hermitian metric compatible with the
skew diagonal symmetric form $[\,\ ,\ ]$ on $\C^{2n}$.  This metric
induces metrics on (the complex points of) all the vector bundles
$E_i$ and the line bundles $L_i=E_{n+1-i}/E_{n-i}$, for $1\leq i\leq
n$. We thus obtain hermitian vector bundles $\ov{E}_i$ and line
bundles $\ov{L}_i$ and, following \cite{GS2}, their arithmetic Chern
classes $\wh{c}_k(\ov{E}_i)\in \wh{\CH}^k(\X)$ and
$\wh{c}_1(\ov{L}_i)\in \wh{\CH}^1(\X)$. Set $\wh{x}_i =
-\wh{c}_1(\ov{L}_i)$ and for any $w\in \wt{W}_n$, define
\[
\wh{\DS}_w := \DS_w(\wh{x}_1,\ldots,\wh{x}_n)\in \wh{\CH}^{\ell(w)}(\X).
\]
The unique map of abelian groups
\begin{equation}
\label{splittingmap}
\epsilon\,:\,\CH(\X)\ra \wh{\CH}(\X)
\end{equation}
sending the Schubert class $\DS_w(\rX_n)$ to $\wh{\DS}_w$ for all
$w\in \wt{W}_n$ splits the exact sequence (\ref{seseq}). We thus
obtain an isomorphism of abelian groups
\begin{equation}
\label{bigiso}
\wh{\CH}(\X)\cong \CH(\X)\oplus \wt{\A}(\X_{\R}).
\end{equation}

\subsection{Computing arithmetic intersections}
\label{compaa}
We now describe an effective procedure for computing arithmetic Chern
numbers on the orthogonal flag variety $\X$, parallel to \cite[\S
4.3]{T5}.  Let $c_k(\ov{E}_i)$ and $c_1(\ov{L}_i)$ denote the Chern
forms of $\ov{E_i(\C)}$ and $\ov{L_i(\C)}$, respectively. In the
sequel we will identify these with their images in $\wh{\CH}(\X)$
under the inclusion $a$.  Let $x_i = -c_1(\ov{L}_i)$ for $1\leq i \leq
n$.

We begin with the short exact sequence
\[
\ov{\E}_{\OG}\ :\ 0 \to \ov{E}_n \to \ov{E} \to \ov{E}_n^* \to 0
\]
where $E_n$ denotes the tautological maximal isotropic subbundle of
$E$ over $\X$. Let $\wt{c}(\ov{\E}_{\OG})\in \wt{\A}(\X_\R)$ be
the {\em Bott-Chern form} \cite{BC, GS2} associated to $\ov{\E}_{\OG}$
for the total Chern class. This form may be computed using
\cite[Prop.\ 3]{T1}, which gives
\begin{equation}
\label{ses}
\wt{c}(\ov{\E}_{\OG}) = \sum_{k=1}^{n-1}(-1)^k\H_kp_k(\ov{E}_n^*).
\end{equation}
Here $p_r(\ov{E}_n^*) = (-1)^r\Tr((K_{E_n})^r)$ denotes the $r$-th
power sum form of $\ov{E}_n^*$, while $\H_r = 1+\frac{1}{2}+\cdots +
\frac{1}{r}$ is a harmonic number. Furthermore, by \cite[Thm.\
4.8(ii)]{GS2}, we have an equation
\begin{equation}
\label{key0}
\wh{c}(\ov{E}_n)\, \wh{c}(\ov{E}^*_n) = 1 + \wt{c}(\ov{\E}_{\OG})
\end{equation}
in $\wh{\CH}(\X)$. 

Consider the hermitian filtration
\[
\ov{\E} \ : \ 0 = \ov{E}_0 \subset \ov{E}_1 \subset \cdots
\subset \ov{E}_n.
\]
Let $\wt{c}(\ov{\E})\in \wt{\A}(\X_\R)$ be the Bott-Chern form of the
hermitian filtration $\ov{\E}$ corresponding to the total Chern class,
as defined in \cite{T2}. According to \cite[Thm.\ 2]{T2}, we have
\begin{equation}
\label{key1}
\prod_{i=1}^n(1-\wh{x}_i) = \wh{c}(\ov{E}_n) + \wt{c}(\ov{\E}).
\end{equation}
If $\wt{c}(\ov{\E}) = \sum_i\alpha_i$ with $\alpha_i\in
\wt{\A}^{i,i}(\X_\R)$ for each $i$, then define $\wt{c}(\ov{\E}^*) =
\sum_i(-1)^{i+1}\alpha_i$. This gives the dual equation
\begin{equation}
\label{key2}
\prod_{i=1}^n(1+\wh{x}_i) = \wh{c}(\ov{E}^*_n) + \wt{c}(\ov{\E}^*).
\end{equation}

The abelian group  $\wt{\A}(\X_{\R})=\Ker\zeta$ is an ideal of $\wh{\CH}(\X)$
such that for any hermitian vector bundle $\ov{F}$ over $\X$ and 
$\eta,\eta'\in \wt{\A}(\X_{\R})$, we have 
\begin{equation}
\label{modstr}
\wh{c}_k(\ov{F}) \cdot \eta = c_k(\ov{F}) \wedge \eta 
\qquad \text{and} \qquad
\eta\cdot \eta' = (dd^c\eta)\wedge\eta'.
\end{equation}
We now multiply (\ref{key1}) with (\ref{key2}) and combine the result with
(\ref{key0}) to obtain
\begin{equation}
\label{key}
\prod_{i=1}^n(1-\wh{x}^2_i) = 1 + \wt{c}(\ov{\E}, \ov{\E}^*), 
\end{equation}
where 
\begin{equation}
\label{keypart}
\wt{c}(\ov{\E}, \ov{\E}^*) = \wt{c}(\ov{\E}_{\OG}) + \wt{c}(\ov{\E})\wedge
c(\ov{E}^*_n) + \wt{c}(\ov{\E}^*)\wedge c(\ov{E}_n) +
(dd^c\wt{c}(\ov{\E}))\wedge \wt{c}(\ov{\E}^*).
\end{equation}
By pulling back \cite[Eqn.\ (6)]{T4} to $\X$, we get the equation
\begin{equation}
\label{keyog}
\wh{c}(\ov{E}^*_n)=\frac{1}{2}\H_{n-1}c_{n-1}(\ov{E}^*_n).
\end{equation}
Equating the top degree terms in (\ref{key2}) and (\ref{keyog}) gives
\begin{equation}
\label{key3}
\wh{x}_1\cdots\wh{x}_n  = \frac{1}{2}\H_{n-1}c_{n-1}(\ov{E}^*_n) + 
\wt{c}_n(\ov{\E}^*).
\end{equation}

In \cite{T2}, it is shown that $\wt{c}(\ov{\E})$ is a polynomial in
the entries of the matrices $K_{E_i}$ and $K_{L_i}$, $1\leq i \leq n$,
with {\em rational} coefficients. Using this, (\ref{ses}), and
(\ref{keypart}), we can express the differential form $\wt{c}(\ov{\E},
\ov{\E}^*)$ as a polynomial in the entries of the matrices $K_{E_i}$
and $K_{L_i}$ with rational coefficients.  On the other hand,
Proposition \ref{grsprop} gives explicit formulas for all these
curvature matrices in terms of $\SO(2n)$-invariant differential forms
on $\X(\C)$. Since we are using the skew diagonal symmetric form to
define the Lie groups here, the formulas in \S \ref{hvb} have to be
changed accordingly. The matrix realization of the Lie algebra
${\mathfrak s}{\mathfrak o}(2n,\C)$ in this case is given in \cite[\S
1.2, \S 2.3]{GW}, while the basis elements of $\mathfrak{h}$ should be
ordered as in \cite[(2.20)]{BH}. The indices $(i,j)$ and $(p,q)$ in
Proposition \ref{grsprop} are then replaced by $(n+1-j,n+1-i)$ and
$(n+1-q,n+1-p)$, respectively. Recalling that $L_i=E_{n+1-i}/E_{n-i}$,
we obtain the identities
\begin{align*}
x_1 &= -\Om_{12} - \Om_{13}- \cdots - \Om_{1n} + \Om^{12} + \Om^{13} +
\cdots + \Om^{1n} \\ 
x_2 &= \hspace{0.27cm} \Om_{12} - \Om_{23} - \cdots - \Om_{2n} +
\Om^{12} + \Om^{23} + \cdots + \Om^{2n} \\ 
&\quad \qquad \qquad \ \ \vdots
\quad \qquad \qquad \vdots \quad \qquad \qquad \vdots \\ 
x_n &= \hspace{0.27cm} \Om_{1n} +\Om_{2n} + \cdots + \Om_{n-1,n} 
+ \Om^{1n} + \Om^{2n} + \cdots + \Om^{n-1,n}
\end{align*}
in $\A^{1,1}(\X_\R)$. We also deduce the next result.
\begin{prop}
\label{compprop}
We have $\wt{c}_1(\ov{\E})=\wt{c}_1(\ov{\E}, \ov{\E}^*)=0$, \
$\dis\wt{c}_2(\ov{\E})=-\sum_{i<j}\Om_{ij}$, \ and 
\[
\wt{c}_2(\ov{\E}, \ov{\E}^*) =   
-2\,\sum_{i<j}\Om_{ij}-2\,\sum_{p<q} \Om^{pq}.
\]
\end{prop}
\begin{proof}
The argument is the same as the proof of \cite[Prop.\ 4]{T5}. 
\end{proof}

\medskip
Let $h(\rX_n)$ be a homogeneous polynomial in the ideal $J_n$ of \S
\ref{soflag}. We give an effective algorithm to compute the
arithmetic intersection $h(\wh{x}_1,\ldots,\wh{x}_n)$ as a class in
$\wt{\A}(\X_\R)$. First, we decompose $h$ as a sum $h(\rX_n)=\sum_i
e_i(\rX_n^2)f_i(\rX_n)+e_n(\rX_n)g(\rX_n)$ for some polynomials $f_i$ and $g$. 
Equation (\ref{key}) implies that
\begin{equation}
\label{simple}
e_i(\wh{x}_1^2,\ldots,\wh{x}_n^2) = (-1)^i \,\wt{c}_{2i}(\ov{\E}, \ov{\E}^*)
\end{equation}
for $1\leq i\leq n$. Using this, (\ref{key3}), and (\ref{modstr}), 
we see that 
\begin{align*}
h(\wh{x}_1,\wh{x}_2,\ldots\wh{x}_n) &= \sum_{i=1}^n
(-1)^i\,\wt{c}_{2i}(\ov{\E},\ov{\E}^*)\wedge f_i(x_1,\ldots,x_n) \\
& + \left(\frac{1}{2}\H_{n-1}c_{n-1}(\ov{E}^*_n) + 
\wt{c}_n(\ov{\E}^*)\right)\wedge g(x_1,\ldots,x_n)
\end{align*}
in $\wh{\CH}(\X)$. By the previous analysis, we can write the right
hand side of the above equation as a polynomial in the $x_i$ and the
entries of the matrices $K_{E_i}$ for $1\leq i \leq n$, with rational
coefficients, which is (the class of) an explicit $\SO(2n)$-invariant
differential form in $\wt{\A}(\X_\R)$.

Let $\wh{\deg}:\wh{\CH}^{n^2-n+1}(\X)\to\R$ denote the
arithmetic degree map of \cite{GS1}.

\begin{thm}
\label{mainthm} For any nonnegative integers  
$k_1,\ldots,k_n$ with $\sum k_i=n^2-n+1$, the arithmetic Chern number
$\dis \wh{\deg}(\wh{x}_1^{k_1}\wh{x}_2^{k_2}\cdots\wh{x}_n^{k_n}) $ is
a rational number.
\end{thm}
\begin{proof}
Since $\sum k_i= \dim{\X}=n^2-n+1$, the monomial 
$\x_1^{k_1}\cdots \x_n^{k_n}$ lies in the ideal $J_n$.
We therefore obtain 
\[
\wh{x}_1^{k_1}\wh{x}_2^{k_2}\cdots\wh{x}_n^{k_n}=
r\,\Omega
\]
for some $r\in \Q$, where $\Om$ is the top invariant form of \S
\ref{hvb}. Using the computation at the end of \S \ref{hvb}, it
follows that
\[
\wh{\deg}(\wh{x}_1^{k_1}\wh{x}_2^{k_2}\cdots\wh{x}_n^{k_n})=
\frac{r}{2}\prod_{k=1}^{n-1}\frac{2}{(2k)!}. 
\]
\end{proof}

The flag variety $\X$ has a natural pluri-Pl\"{u}cker embedding $j$ in 
projective space. The morphism $j$ is defined as the composite of 
the natural inclusion of $\X$ into the variety parametrizing
all partial flags
\[
0=E_0\subset E_1\subset \cdots \subset E_n \subset E_{2n}=E
\]
with $\dim(E_i)=i$ for each $i$, followed by the pluri-Pl\"{u}cker
embedding of the latter type A flag variety into projective space. Let
$\ov{\O}(1)$ denote the canonical line bundle over projective space,
equipped with its canonical metric (so that $c_1(\ov{\O}(1))$ is the
Fubini-Study form). Following \cite{GS1, Fa, BoGS}, the 
{\em projective height} of $\X$ relative to
$\ov{\O}(1)$ is given by
\[
h_{\ov{\O}(1)}(\X)
=\wh{\deg}\left(\wh{c}_1(\ov{\O}(1))^{n^2-n+1}\vert \ \X \right).
\]
Using Theorem \ref{mainthm} and arguing as in \cite[\S 4.6]{T5}, we
conclude that the projective height $h_{\ov{\O}(1)}(\SO_{2n}/B)$ is a
rational number. The height formula of Kaiser and K\"ohler \cite{KK}
provides a different proof of this fact. Relating these two approaches
to computing the height to each other seems rather difficult; some
first steps in this direction are taken in \cite{T3, T4}.

\subsection{Arithmetic Schubert calculus}
\label{asc}

For any partition $\l\in\G_n$ and $\om\in S_n$, define 
\[
\wh{\DS}_{\l,\om} = \DS_{\l,\om}(\wh{x}_1,\ldots,\wh{x}_n). 
\]

If $\l\in\G_n\ssm\F_{n-1}$, let $r_\l$ be the largest repeated part of
$\l$, and let $\ov{\l}$ be the partition obtained from $\l$ by
deleting two (respectively, one) of the parts $r_{\l}$ if 
$r_\l <n$ (respectively, if $r_\l=n$).
If $r_{\l}<n$, then properties (b), (c) in \S \ref{definitions},
(\ref{modstr}), and (\ref{simple}) imply that 
\[
\wh{\DS}_{\l,\om} = \wh{\DS}_{\ov{\l},\om}
\wt{P}_{r_\l,r_\l}(\wh{x}^2_1,\ldots,\wh{x}^2_n) 
= \frac{(-1)^{r_\l}}{4}\DS_{\ov{\l},\om}(x_1,\ldots,x_n)\wedge
\wt{c}_{2r_\l}(\ov{\E}, \ov{\E}^*).
\]
If $r_{\l}=n$, then property (f) in \S \ref{definitions},
(\ref{modstr}), and (\ref{key3}) give
\[
\wh{\DS}_{\l,\om} = \wh{\DS}_{\ov{\l},\om}
\wt{P}_n(\wh{x}_1,\ldots,\wh{x}_n) 
= \frac{1}{2}\DS_{\ov{\l},\om}(x_1,\ldots,x_n)\wedge
(\frac{1}{2}\H_{n-1}c_{n-1}(\ov{E}^*_n) + 
\wt{c}_n(\ov{\E}^*)).
\]
Since $\wh{\DS}_{\l,\om} \in a(\wt{\A}(\X_\R))$ whenever 
$\l\in \G_n\ssm\F_{n-1}$,
we will denote these classes by $\wt{\DS}_{\l,\om}$.  
The next theorem uses the basis 
of orthogonal Schubert polynomials to compute 
arbitrary arithmetic intersections in
$\wh{\CH}(\X)$ with respect to the splitting (\ref{bigiso}) induced 
by (\ref{splittingmap}).

\begin{thm}
\label{SOring}
Any element of the arithmetic Chow ring $\wh{CH}(\X)$ can be
expressed uniquely in the form $\dis \sum_{w\in
\wt{W}_n}a_w\wh{\DS}_w+\eta$, where $a_w\in A$ and
$\eta\in\wt{\A}(\X_{\R})$. For $u,v\in \wt{W}_n$ we have
\begin{equation}
\label{maineq}
\wh{\DS}_u\cdot\wh{\DS}_v=\sum_{w\in \wt{W}_n}d_{uv}^w\,\wh{\DS}_w+
\sum_{{\l\in \G_n\ssm\F_{n-1}}\atop{\om\in
S_n}}d_{uv}^{\l\om}\,\wt{\DS}_{\l,\om},
\end{equation}
\[
\wh{\DS}_u\cdot \eta=\DS_u(x_1,\ldots,x_n)\wedge\eta,
\ \ \ \ \text{and}
\ \ \ \ \eta\cdot \eta'=(dd^c\eta)\wedge\eta',
\]
where $\eta$, $\eta'\in\wt{\A}(\X_{\R})$ and the integers $d_{uv}^w$,
$d_{uv}^{\l\om}$ are as in {\em (\ref{structeq})}.
\end{thm}
\begin{proof} 
The first statement is a consequence of the splitting (\ref{bigiso}).
Equation (\ref{maineq}) is a consequence of the formal identity
(\ref{structeq}) and our definitions of $\wh{\DS}_w$ and
$\wt{\DS}_{\l,\om}$. The remaining assertions follow from the
structure equations (\ref{modstr}).
\end{proof}

\medskip
We remark that one can refine Theorem \ref{SOring} by replacing
$\wh{\CH}(\X)$ with the invariant arithmetic Chow ring
$\wh{\CH}_{\inv}(\X)$. Following \cite[\S 4.5]{T5}, the ring
$\wh{\CH}_{\inv}(\X)$ is obtained by substituting the space $\A(\X_\R)$
with a certain subspace of the space of all $\SO(2n)$-invariant
differential forms on $\X(\C)$. We leave the details to the reader.

\end{document}